\documentclass[letterpaper, 10 pt, conference]{ieeeconf}  

\IEEEoverridecommandlockouts                              

\usepackage{graphics} 
\usepackage{epsfig} 
\usepackage{mathptmx} 
\usepackage{times} 
\usepackage{amsmath} 
\usepackage{amssymb}  
\usepackage{verbatim}
\usepackage[ruled,linesnumbered]{algorithm2e}
\RestyleAlgo{ruled}

\newtheorem{theorem}{Theorem}
\newtheorem{proposition}{Proposition}

\newtheorem{remark}{Remark}
\newtheorem{corollary}{Corollary}

\newtheorem{definition}{Definition}

\newtheorem{example}{Example}

\usepackage{color}

\newcommand{\bR} { {\mathbb R}}
\newcommand{\bC} { {\mathbb C}}

\title{\LARGE \bf
Partial Eigenvalue Assignment for Nonlinear Systems
}

\author{Shang Wang, Xiaodong Cheng, Yu Kawano and Peter van Heijster
\thanks{This work of Kawano was supported in part by JST FOREST Program Grant Number JPMJFR222E and JSPS KAKENHI Grant Number 21H04875.}
\thanks{S. Wang, X. Cheng and P. van Heijster are with Mathematical and Statistical Methods Group (Biometris), Department of Plant Science, Wageningen University \& Research, 6700 AA Wageningen, The Netherlands.
{\tt\small \{shang.wang,xiaodong.cheng,peter.vanheijster\}
@wur.nl}}
\thanks{Y. Kawano is with the Graduate School of Advanced Science and Engineering, Hiroshima University, Higashi-Hiroshima 739-8527, Japan. {\tt\small ykawano@hiroshima-u.ac.jp}}
}

\begin{document}

\maketitle
\thispagestyle{empty}
\pagestyle{empty}

\begin{abstract}
    In this paper, we study control design methods for assigning a subset of nonlinear right or left eigenvalues to a specified set of scalar-valued functions via nonlinear Sylvester equations. This framework can be viewed as a generalization of partial linear eigenvalue assignment (also referred to as partial pole placement) for linear systems. First, we propose a method for partial nonlinear right eigenvalue assignment via state feedback using a nonlinear Sylvester equation and a condition for preserving an open-loop nonlinear right eigenvalue. This method can be applied to partial stabilization of nonlinear systems. Then, as the dual problem, we present a method for partial nonlinear left eigenvalue assignment via the dual nonlinear Sylvester equation and a condition for preserving an open-loop nonlinear left eigenvalue, which can be applied to partial observer design for nonlinear system.

\end{abstract}
Index terms: Nonlinear systems, Nonlinear eigenvalues, Pole placement, Stabilization, Observer design

\section{Introduction}
\label{sec:introduction}
Eigenvalue assignment plays a fundamental role in controller design, since the eigenvalues of a dynamical system determine various critical properties, e.g., stability, convergence rate, and transient responses. In the context of linear systems, eigenvalue or pole placement via state feedback is a well-established technique, with methods such as the Ackermann's formula via the coefficient matching of characteristic polynomials \cite{ackermann1972entwurf,nordstrom1997multi} and methods using the linear Sylvester equation \cite{BHATTACHARYYA1982pole}. 
While full eigenvalue assignment aims to prescribe the entire closed-loop spectrum, partial eigenvalue placement has emerged as an important strategy when the control objective is to shape only a subset of closed-loop eigenvalues. This approach offers flexibility to address secondary objectives, such as disturbance rejection and optimal control, by exploiting the degrees of freedom associated with the unspecified spectrum. For example, \cite{ram2013multiple} shows the application of partial eigenvalue placement in designing controllers for vibration suppression, which is extended to enable simultaneous partial and regional eigenvalue placement in \cite{datta2010partial,datta2012partial}. Our recent work \cite{wang2024partial} develops a systematic framework for partial eigenvalue placement in linear systems, which shows how these unused degrees of freedom can be leveraged for auxiliary design objectives.


Many dynamical systems in real-world applications are inherently nonlinear, for which linear eigenvalue assignment methods are often applied after linearization around equilibrium points, see e.g., \cite{rao1981advanced}. However, the linearization-based method may fail to capture essential nonlinear dynamics, potentially giving misleading conclusions about the system's behavior.
Furthermore, controllers designed based on linearized models may result in performance degradation or even instabilities when strong nonlinearities dominate \cite{zhu2002pole}. To address this issue, several methods have been proposed to extend eigenvalue assignment to nonlinear systems without relying on linearization. For instance, \cite{zhu2002pole} introduces a nonlinear pole placement method for systems whose outputs are polynomial functions of past inputs and outputs, and \cite{tomas2004pole} presents a pole placement method based on approximating nonlinear dynamics with a sequence of linear time-varying models.
However, these approaches are restricted to a particular class of nonlinear systems and rely on strong structural assumptions or approximations.

In contrast, this paper aims to provide a more fundamental approach  involving the concept of nonlinear eigenvalues of nonlinear systems, which neither applies linearization methods nor requires structural assumptions or approximations. 
The works \cite{KAWANO2013606,halas2013definition} introduce a concept of nonlinear eigenvalues and eigenvectors.
This definition is formalized in \cite{kawano2015stability,kawano2017pbh,kawano2017nonlinear} to include both left and right nonlinear eigenstructures. 
Importantly, these nonlinear eigenvalues can also indicate fundamental properties of nonlinear systems, such as equilibrium stability and transient behavior, similar to the linear case. For instance, \cite{kawano2015stability} establishes an asymptotic stability criterion based on the nonlinear eigenstructure of a diagonalizable system. This naturally raises the question: \textit{can we design controllers to assign nonlinear eigenvalues and thereby shape the dynamic behavior of nonlinear systems?}

To the best of our knowledge, general methods for nonlinear eigenvalue assignment remain largely unexplored.
A related but distinct line of work involves immersion-based observer and controller design. In \cite{kazantzis1998nonlinear}, an observer design method is proposed that assigns the eigenvalues of the error dynamics for nonlinear systems by immersing them into a linear system via a state transformation. This method, however, requires the existence of a suitable immersion mapping and does not enable the assignment of nonlinear eigenvalues defined in \cite{kawano2015stability,kawano2017pbh,kawano2017nonlinear}. 
Moreover, \cite{astolfi2003immersion} presents a nonlinear stabilization method by immersing an exo-system into the closed-loop dynamics. While the asymptotic behavior of the exo-system is inherited by the closed-loop system under projection, this approach is not explicitly designed for shaping the nonlinear eigenstructure of nonlinear systems. Note that both methods in \cite{kazantzis1998nonlinear} and \cite{astolfi2003immersion} utilize the nonlinear Sylvester equation derived from the immersion condition. This equation also serves as a key tool in addressing several well-known problems in nonlinear control, e.g., the output regulation problem \cite{isidori1990output}.



These limitations motivate us to develop a constructive framework that can explicitly change the nonlinear eigenstructure of nonlinear dynamical systems. The main contributions are summarized as follows. First, we develop a constructive method for assigning a subset of the nonlinear right eigenvalues of a nonlinear system via state feedback. The approach is based on solving a nonlinear Sylvester equation that arises from the immersion condition in \cite{astolfi2003immersion}, combined with the concept of nonlinear eigenvalues in \cite{kawano2015stability,kawano2017pbh,kawano2017nonlinear}. Furthermore, we derive a condition to preserve selected open-loop nonlinear eigenvalues and eigenvectors in the closed-loop system, thereby enabling the achievement of a secondary control objective. 
Second, the framework is extended to address the dual problem of partial nonlinear left eigenvalue assignment via the dual nonlinear Sylvester equation. The results are applicable to the synthesis of partial observer and partial stabilization by output-feedback control.

The remainder of this paper is organized as follows. Section~\ref{sec:pre} introduces necessary preliminaries for the subsequent developments. In Section~\ref{sec:example}, an example is presented to motivate our research of nonlinear eigenvalue assignment. Section~\ref{sec:main} outlines the framework for partial nonlinear right eigenvalue assignment for nonlinear systems, and Section~\ref{sec:dual} extends the results to partial nonlinear left eigenvalue assignment. Finally, conclusion remarks are made in Section~\ref{sec:conclusion}.

\subsubsection*{Notation} 
The fields of real numbers and complex numbers are denoted as $\bR$ and $ \bC$, respectively. 
The zero matrix of dimension $n\times m$ is denoted by $0_{n\times m}$.
The transpose of a matrix $A$ is denoted by $A^\top$.
For a complex matrix $A$, its real part and imaginary part are denoted by $\mathrm{Re}(A)$ and $\mathrm{Im}(A)$, respectively.
The set of eigenvalues of a square matrix $A$ is denoted by $\sigma (A)$.
The set of functions which are at least $k$ times continuously differentiable on the domain of definition is denoted by $C^k$. 
The Lie bracket of two differentiable vector fields $v,s:\bR^\nu \to  \bC^\nu$ is denoted by
\begin{align*}
[v(w),s(w)] := \frac{\partial s(w)}{\partial w}v(w) - \frac{\partial v(w)}{\partial w} s(w).
\end{align*}




\section{Preliminaries}\label{sec:pre}
\subsection{Review of Partial Eigenvalue Assignment for Linear Systems}
First, we briefly review our previous work \cite{wang2024partial} on partial eigenvalue assignment via Sylvester equations for linear systems. Consider the following controllable linear time-invariant (LTI) system
\begin{align}
\label{eq:lsys}
    \dot{x}(t) = Ax(t)+Bu(t), \quad t\geq 0,
\end{align}
where $A \in \bR^{n \times n}$ and $B \in \bR^{n\times m}$ with $m \leq n$. For the sake of simplicity, the time variable $t$ is omitted hereafter. 

Applying the state feedback $u=Kx$, the closed-loop system becomes
$
    \dot{x} = (A + B K)x.
$
To specify a part of closed-loop eigenvalues, we introduce two matrices $L\in\bR^{m\times \nu}$ and $S\in\bR^{\nu\times \nu}$, $\nu \le n$. If $\sigma(S)\cap\sigma(A)=\emptyset$ and $(L,S)$ is observable, then the following linear Sylvester equation has a unique and full rank solution $\Pi\in\bR^{n\times \nu}$, e.g., \cite{bartels1972solution,souza1981controllability}:
\begin{align}\label{eq:lsyl}
    A\Pi+BL=\Pi S.
\end{align}

Utilizing $\Pi$, partial eigenvalue assignment can be achieved as follows.
\begin{proposition}\cite[Theorem~1]{wang2024partial}\label{prop:linear}
If there exist $K \in \bR^{m \times n}$, $L \in \bR^{m\times \nu}$, and $S\in\bR^{\nu\times \nu}$ such that the linear Sylvester equation~\eqref{eq:lsyl} admits a unique and full rank solution $\Pi\in\bR^{n\times \nu}$, and $L = K \Pi$ holds, then we have $\sigma(S)\subset \sigma(A + BK)$.
    
\end{proposition}

\begin{remark}
    If $\sigma(S)\cap\sigma(A)\neq\emptyset$, the Sylvester equation \eqref{eq:lsyl} may have infinitely many solutions. In this case, if one solution $\Pi$ is of full rank, partial eigenvalue assignment can still be achieved with $K$ satisfying $L=K\Pi$.
    
\end{remark}

One can utilize the flexibility of the $n-\nu$ unspecified eigenvalues to achieve different control goals. For example, the unspecified eigenvalues can be designed as a subset of open-loop eigenvalues, e.g., \cite[Corollary~2]{wang2024partial}.


\subsection{Nonlinear Eigenvalues for Nonlinear Systems}
The concepts of eigenvalues and eigenvectors of linear systems have been generalized to nonlinear systems in, e.g., \cite{KAWANO2013606,halas2013definition,kawano2015stability,kawano2017pbh,kawano2017nonlinear}. In this subsection, we tailor these definitions to accommodate with the framework of partial nonlinear eigenvalue placement.

Consider the nonlinear autonomous system described by
\begin{align}
\label{eq:exo}
    \dot{w} = s(w),
\end{align}
where $s:\bR^\nu \to \bR^\nu$ is of class $C^1$ such that $s(0) = 0$. 

Its nonlinear right eigenvector is defined as a one-dimensional invariant distribution with respect to \eqref{eq:exo} as follows.

\begin{definition}
\label{def:reig}
    Let $D_w\subseteq\bR^\nu$.
    The class $C^0$ function $\lambda:D_w \to  \bC$ and class $C^1$ function $v: D_w \to  \bC^\nu \setminus\{0_{\nu \times 1}\}$ are respectively said to be a \emph{nonlinear right eigenvalue} and \emph{nonlinear right eigenvector} of \eqref{eq:exo} on $D_w$ if
    \begin{align}
    \label{eq:reig}
    [v(w),s(w)] =\lambda(w) v(w), 
    \end{align}
    holds for all $w \in D_w$.
    
\end{definition}

Similarly, the nonlinear left eigenvector is defined as a one-dimensional invariant codistribution with respect to \eqref{eq:exo} as follows.

\begin{definition}
\label{def:leig}
    Let $D_w\subseteq\bR^\nu$.
    The class $C^0$ function $\lambda:D_w \to \bC$ and class $C^1$ function $v:D_w \to  \bC^\nu\setminus\{0_{\nu \times 1}\}$ are respectively said to be a \emph{nonlinear left eigenvalue} and \emph{nonlinear left eigenvector} of \eqref{eq:exo} on $D_w$ if
    \begin{align}
    \label{eq:leig}
    v^\top (w)\frac{\partial s(w)}{\partial w} + \left(\frac{\partial v(w)}{\partial w}s(w)\right)^\top = \lambda(w)v^\top (w)
    \end{align}
    holds for all $w \in D_w$.
    
\end{definition}

Differently from \cite[Definition~2,3]{kawano2015stability}, nonlinear eigenvalues and their associated eigenvectors considered here are not necessarily to be complex analytic and are defined on a subset $D_w\subseteq \bR^\nu$.

\begin{remark}
Definitions~\ref{def:reig} and \ref{def:leig} can be regarded as generalizations of conventional eigenvalues and (left and right) eigenvectors of linear systems, since the conventional linear eigenvalues and eigenvectors of a linear system also satisfy Definitions~\ref{def:reig} and \ref{def:leig}. However, differently from linear eigenvalues, there are infinitely many nonlinear eigenvalues even for linear systems caused by $\delta_f$-conjugacy \cite[Section II.B]{kawano2017nonlinear}. In the set of a $\delta_f$-conjugate eigenvalues, an important eigenvalue for stability analysis is the one which corresponds to an integrable eigenvector~\cite{kawano2015stability}.~

\end{remark}

\section{Motivating Example}
\label{sec:example}

In this section, we present a motivating example to illustrate the need for partial nonlinear eigenvalue assignment in nonlinear systems and to highlight the limitations of linearization-based analysis. Consider the following two-dimensional nonlinear system
\begin{equation}
    \label{eq:ex_sys_1}
    \dot{x} =  \begin{bmatrix}
        -x_1+2x_1x_2\\-x_2-2x_1^2+2x_2^2
    \end{bmatrix}+\begin{bmatrix}
        x_1x_2\\-x_1^2+x_2^2
    \end{bmatrix}u.
\end{equation}
The open-loop nonlinear right eigenvalues and eigenvectors of this system, based on the definition in \cite{kawano2015stability}, are given by:
\begin{align*}
\begin{aligned}
\lambda_{o,1}(x) &= -1+2x_2+2ix_1,
& 
v_{o,1}(x) &= \begin{bmatrix}
-i\\1
\end{bmatrix}\\ 
\lambda_{o,2}(x) &= -1+2x_2-2ix_1, 
& 
v_{o,2}(x) &= \begin{bmatrix}
    i\\1
\end{bmatrix},
\end{aligned}
\end{align*}
where $i$ denotes the imaginary unit. We now shape the nonlinear eigenvalues by designing a simple feedback controller of the form $u=b \geq 0$, with $b$  a constant gain, which results in the closed-loop system as
\begin{equation}
\label{eq:ex_sys_1_cl}
    \dot{x} =  \begin{bmatrix}
        -x_1+(2+b)x_1x_2\\-x_2-(2+b)(x_1^2-x_2^2)
    \end{bmatrix}.
\end{equation}
The linearized system of \eqref{eq:ex_sys_1_cl} at the origin does not depend on $b$, and thus provides no insight into how the control input affects the system dynamics. Instead, we examine the nonlinear eigenvalues of the closed-loop system, which are:
\begin{align*}
\begin{aligned}
\lambda_{c,1}(x) &= -1+(2+b)x_2+i(2+b)x_1,
& 
v_{c,1}(x) &= \begin{bmatrix}
-i\\1
\end{bmatrix}\\ 
\lambda_{c,2}(x) &= -1+(2+b)x_2-i(2+b)x_1, 
& 
v_{c,2}(x) &= \begin{bmatrix}
    i\\1
\end{bmatrix}.
\end{aligned}
\end{align*}
Figure~\ref{fig:1} shows the system responses for different values of $b$ under the initial condition $x(0) = \begin{bmatrix}
    1&1
\end{bmatrix}^\top$. It can be inferred that increasing $b$
generally leads to larger overshoots and shorter settling time in this example.
This suggests that, analogously to eigenvalues of linear systems, nonlinear eigenvalues can serve as an indicator of transient behavior. Furthermore, as shown in \cite[Theorem~5]{kawano2015stability}, the stability of an equilibrium point in a nonlinear system can also be inferred from its nonlinear eigenvalues. Therefore, reshaping the nonlinear eigenstructure through feedback offers an effective way to change both stability and transient performance of a nonlinear system.

This example motivates the need for a systematic method to assign nonlinear eigenvalues, which is a problem we refer to as nonlinear eigenvalue assignment. In the next section, we develop a detailed design method for this control objective.



\begin{figure}[t]
    \centering
    \includegraphics[width=0.45\textwidth]{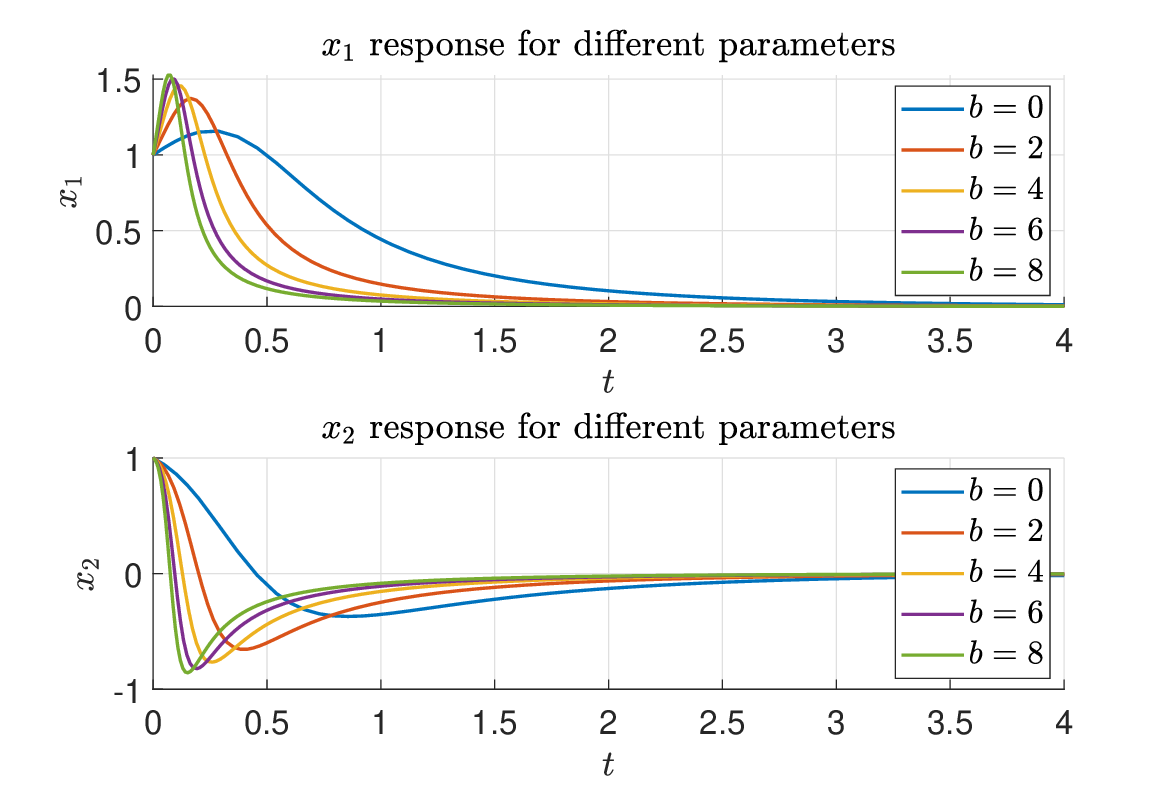}
    \caption{Response of the system \eqref{eq:ex_sys_1} for different values of $b$}
    \label{fig:1}
\end{figure}


%

    
%




\section{Partial Right Eigenvalue Assignment}
\label{sec:main}

In this section, we extend the framework of partial nonlinear eigenvalue assignment to nonlinear systems, utilizing nonlinear right eigenvalues and their associated right eigenvectors. In particular, we derive sufficient conditions for partial nonlinear eigenvalue assignment based on a nonlinear Sylvester equation arising from an immersion condition. Additionally, we derive a condition for preserving specified open-loop nonlinear eigenvalues and eigenvectors in the closed-loop system. Based on the stability results in~\cite{kawano2015stability}, we also discuss partial stabilization using the proposed method.


\subsection{Main Results}
\label{sec:main_results}

Consider the following nonlinear input-affine system:
\begin{align}
\label{eq:sys_x}
    \dot{x} = f(x)+g(x)u,
\end{align}
where $f:\bR^n \to \bR^n$ and $g:\bR^n \to \bR^{n \times m}$ are of class $C^1$ such that $f(0) = 0$. For nonlinear state feedback $u=k(x)$, the closed-loop system is
\begin{align}\label{eq:sys_cl}
\dot{x}=\tilde{f}(x):=f(x)+g(x)k(x),
\end{align}
where $k:\bR^n \to \bR^m$ is of class $C^1$ such that $k(0) = 0$.

Next, we introduce a nonlinear extension of the Sylvester equation~\eqref{eq:lsyl}.
Let $l:\bR^\nu \to \bR^m$ be of class $C^1$ such that $l(0) = 0$. Then, the nonlinear Sylvester equation, e.g.,~\cite{astolfi2003immersion,scarciotti2024interconnection}, is the following nonlinear partial differential equation (PDE) with respect to $\pi (w)$:
\begin{align}
\label{eq:syl}
f(\pi(w)) + g(\pi(w))l(w)=\frac{\partial\pi (w)}{\partial w}s(w).
\end{align}

As the first main result, we extend Proposition~\ref{prop:linear} for partial eigenvalue assignment to nonlinear systems as follows.

\begin{theorem}
\label{thm:main}
    Suppose that there exist class $C^1$ functions $k:\bR^n \to \bR^m$, $l:\bR^\nu \to \bR^m$, and $s:\bR^\nu \to \bR^\nu$ with $k(0) = 0$, $l(0)=0$, and $s(0)=0$, satisfying the following three conditions:
    \begin{enumerate}
    \item the exo-system \eqref{eq:exo} admits a nonlinear right eigenvalue $\lambda (w)$ and its associated nonlinear right eigenvector $v(w)$ on $\bR^\nu$;
    
    \item the nonlinear Sylvester equation~\eqref{eq:syl} has a class $C^1$ solution $\pi(w)$ on $\bR^\nu$ such that $\pi(0)=0$ and 
    \begin{align}
    \label{eq:pep}
        k(\pi(w)) = l(w)
    \end{align}
    holds on $\bR^\nu$;
    
    \item there exist $\Tilde{\lambda}:\bR^n \to  \bC$ of class $C^0$ and $\Tilde{v}:\bR^n \to  \bC^n$ of class $C^1$ such that 
   \begin{subequations}\label{cond1:main}
    \begin{align}
    \Tilde{\lambda}(\pi(w)) &= \lambda(w) \\
    \tilde{v}(\pi(w)) &= \frac{\partial \pi (w)}{\partial w} v(w)
    \label{cond12:main}
    \end{align}
    \end{subequations}
    hold on $\bR^\nu$, and $\tilde{v}(\pi(w)) \neq 0_{\nu \times 1}$ for all $w \in \bR^\nu$.
    \end{enumerate}
    Then, $\tilde{\lambda}(x)$ is a nonlinear right eigenvalue of the closed-loop system \eqref{eq:sys_cl} on the projective space $\pi (\bR^\nu):=\{x:x=\pi(w),\ w\in\bR^\nu\}$ with its associated nonlinear right eigenvector $\tilde{v}(x)$.
\end{theorem}

\begin{proof}
We have that
\begin{align}\label{pf1:main}
&[\tilde{v}(x),\tilde f(x)]|_{x = \pi (w)} \nonumber\\
&= \frac{\partial \tilde f(\pi (w))}{\partial x} \tilde v(\pi (w)) - \frac{\partial \tilde v(\pi (w))}{\partial x} \tilde f(\pi (w))
\end{align}
with $\tilde{f}$ defined in \eqref{eq:sys_cl}.
Applying~\eqref{cond12:main}, \eqref{eq:pep}, \eqref{eq:sys_cl}, and \eqref{eq:syl} in order, the first term of the right-hand side can be rearranged as
\begin{align*}
&\frac{\partial \tilde f(\pi (w))}{\partial x} \tilde v(\pi (w)) \\
&= \frac{\partial \tilde f(\pi (w))}{\partial x} \frac{\partial \pi (w)}{\partial w} v(w) \\
&=\frac{\partial \tilde f(\pi (w))}{\partial w} v(w) \\
&= \frac{\partial \pi (w)}{\partial w} \frac{\partial s(w)}{\partial w} v(w)
+ \left( \sum_{i=1}^\nu \frac{\partial}{\partial w_i} \left(\frac{\partial \pi (w)}{\partial w}\right) s_i(w) \right) v(w).
\end{align*}
Similarly, applying \eqref{eq:sys_cl}, \eqref{eq:pep}, \eqref{eq:syl}, \eqref{cond12:main} in order, the second term can be rearranged as
\begin{align*}
&\frac{\partial \tilde v(\pi (w))}{\partial x} \tilde f(\pi (w)) \\
&=\frac{\partial \tilde v(\pi (w))}{\partial x} \frac{\partial \pi (w)}{\partial w}s(w) \\
&= \frac{\partial \tilde v(\pi (w))}{\partial w} s(w) \\
&= \frac{\partial \pi (w)}{\partial w} \frac{\partial v(w)}{\partial w} s(w) 
+ \left( \sum_{i=1}^\nu \frac{\partial}{\partial w_i} \left(\frac{\partial \pi (w)}{\partial w}\right) v_i(w) \right) s(w).
\end{align*}
Substituting these two into~\eqref{pf1:main} yields 
\begin{align*}
&[\tilde{v}(x),\tilde f(x)]|_{x = \pi (w)} \\
&= \frac{\partial \pi (w)}{\partial w} [v(w), s(w)] \\
&\qquad +\left( \sum_{i=1}^\nu \frac{\partial}{\partial w_i} \left(\frac{\partial \pi (w)}{\partial w}\right) s_i(w) \right) v(w) \\
&\qquad - \left( \sum_{i=1}^\nu \frac{\partial}{\partial w_i} \left(\frac{\partial \pi (w)}{\partial w}\right) v_i(w) \right) s(w).
\end{align*}
Note that 
\begin{align*}
&\left( \sum_{i=1}^\nu \frac{\partial}{\partial w_i} \left(\frac{\partial \pi (w)}{\partial w}\right) s_i(w) \right) v(w) \\
&= \sum_{i=1}^\nu \sum_{j=1}^\nu \frac{\partial^2 \pi (w)}{\partial w_i w_j} s_i(w) v_j(w)\\
&=\left( \sum_{i=1}^\nu \frac{\partial}{\partial w_i} \left(\frac{\partial \pi (w)}{\partial w}\right) v_i(w) \right) s(w).
\end{align*}
Therefore, it follows from~\eqref{eq:reig} and~\eqref{cond1:main} that
\begin{align*}
[\tilde{v}(x),\tilde f(x)]|_{x = \pi (w)} 
&= \lambda (w) \frac{\partial \pi (w)}{\partial w} v(w)\\
&= (\tilde{\lambda}(x)\tilde{v}(x))|_{x = \pi (w)}.
\end{align*}
This completes the proof.
\end{proof}

\begin{remark}
With a similar proof, Theorem~\ref{thm:main} can be generalized to nonlinear systems that are not necessarily input-affine, i.e., systems of the form $\dot x = \hat{f}(x,u)$  where $\hat{f}:\bR^n\times\bR^m\rightarrow\bR^n$ is of class $C^1$ and $\hat{f}(0,0)=0$.

\end{remark}

\begin{remark}\label{rem:mlt}
Item 3) of Theorem~\ref{thm:main} is a mild condition because $\tilde \lambda (x)$ and $\tilde v(x)$ are arbitrary as long as they satisfy~\eqref{cond1:main}. In fact, there can be multiple $\tilde \lambda (x)$ and $\tilde v(x)$ satisfying~\eqref{cond1:main} if $\nu < n$. If $\lambda (w)$ is constant, i.e., $\lambda (w) = c$ on $\bR^\nu$, then a natural choice would be $\Tilde{\lambda} (x) = c$ on $\bR^n$.

\end{remark}

\begin{remark}
    To the best of our knowledge, there is no easy-to-check condition on the global existence of a solution $\pi(w)$ to a nonlinear Sylvester equation \eqref{eq:syl}. However, the local existence around the origin can be verified by utilizing, e.g., \cite[Auxiliary Proposition]{lyapunov1992general} and \cite[Lemma 1]{isidori1990output}.~
\end{remark}

Theorem~\ref{thm:main} can be applied to partial nonlinear eigenvalue assignment via the following procedure:
\begin{enumerate}
    \item design an exo-system \eqref{eq:exo} having desirable $\nu$ nonlinear right eigenvalues $\lambda_1(w),\dots,\lambda_\nu(w)$ and their associated nonlinear right eigenvectors $v_1(w),\dots,v_\nu(w)$;
    
    \item find $l(w)$ such that the nonlinear Sylvester equation \eqref{eq:syl} has a solution $\pi(w)$ on $\bR^\nu$, where $\partial \pi (w)/\partial w$ is of full rank at each $w \in \bR^\nu$;

    \item find $k(x)$ and $(\tilde \lambda_i (x), \tilde v_i(x))$, $i=1,\dots,\nu$ satisfying \eqref{eq:pep}, \eqref{cond1:main} on $\bR^\nu$, respectively.
\end{enumerate}
Then for all $i=1,\ldots,n$, $\tilde \lambda_i (x)$, $\tilde v_i(x)$ are a pair of nonlinear right eigenvalue and nonlinear right eigenvector of the closed-loop system~\eqref{eq:sys_cl} on $\pi (\bR^\nu)$.

\begin{remark}
    A sufficient condition for $\tilde{v}(\pi(w)) \neq 0_{\nu \times 1}$ is given by $\mathrm{rank}(\partial \pi(w)/\partial w)=\nu$ on $\bR^\nu$. In addition, if $\mathrm{rank}_{\bR^\nu}(\begin{bmatrix}
    v_1(w)&\ldots&v_\nu(w)
    \end{bmatrix})=\nu$ on $\bR^\nu$, then $\mathrm{rank}(\partial \pi(w)/\partial w)=\nu$ on $\bR^\nu$ also guarantees that the $\nu$ closed-loop nonlinear right eigenvectors are linearly independent on $\pi (\bR^\nu)$. This procedure is a natural extension of the linear case, where $\lambda_i = \tilde \lambda_i$, $i=1,\dots,\nu$.
\end{remark}


By partial nonlinear eigenvalue assignment, we only assign at most~$\nu$ nonlinear right eigenvalues of the closed-loop system. When $\nu < n$, the remaining nonlinear eigenvalues can remain unchanged under state feedback if the following condition holds. 

\begin{theorem}\label{thm:inv}
Let $\lambda:\bR^n \to \bC$ of class $C^0$ and $v:\bR^n \to \bC^n\setminus\{0_{n\times 1}\}$ of class $C^1$ be, respectively, a nonlinear right eigenvalue and its associated nonlinear right eigenvector of $\dot x = f(x)$ on a subset $D_x\subset \bR^n$. If $k:\bR^n \to \bR^m$ of class $C^1$ satisfies
    \begin{align}
    \label{eq:inv}
        [v(x), g(x) k(x)] = 0
    \end{align}
    on $D_x$, then $\lambda (x)$ and $v(x)$ are a pair of nonlinear right eigenvalue and nonlinear right eigenvector of the closed-loop system~\eqref{eq:sys_cl} on $D_x$.
\end{theorem}
\begin{proof}
It follows from \eqref{eq:reig}, i.e., $[v(x), f(x)] = \lambda (x) v(x)$ and~\eqref{eq:inv} that
\begin{align*}
&[v(x), f(x) + g(x) k(x)] \\
&= [v(x), f(x)] + [v(x), g(x) k(x)]
= \lambda (x) v(x).
\end{align*}
That completes the proof.
\end{proof}

\begin{remark}
In the linear case, \eqref{eq:inv} becomes $B K v = 0$. If $B$ is of full column rank, this is equivalent to $K v = 0$, which is exactly the condition in \cite[Lemma 1]{wang2024partial}.

\end{remark}

Below, we illustrate Theorems~\ref{thm:main} and~\ref{thm:inv} for partial nonlinear right eigenvalue assignment by an example.

\begin{example}
\label{ex:stab}
Consider the nonlinear system \eqref{eq:sys_x} with
\begin{align}\label{ex:osys}
    f(x) = \begin{bmatrix}
        -x_1-x_1^2/2 +x_1x_2+2x_2-x_2^2\\x_2-x_2^2/2
    \end{bmatrix},\ g(x) = \begin{bmatrix}
        1\\1
    \end{bmatrix}.
\end{align}
nonlinear right eigenvalues and their associated nonlinear right eigenvectors of $\dot x = f(x)$ (also referred to as open-loop eigenvalues and eigenvectors) on $\bR^2$ are
\begin{align*}
\begin{aligned}
\lambda_{o,1}(x) &= -1-x_1+x_2,
& 
v_{o,1}(x) &= \begin{bmatrix}
1\\0
\end{bmatrix}\\ 
\lambda_{o,2}(x) &= 1-x_2, 
& 
v_{o,2}(x) &= \begin{bmatrix}
    1\\1
\end{bmatrix}.
\end{aligned}
\end{align*}

We design $u=k(x)$ such that one of the closed-loop nonlinear right eigenvalues is assigned at $-1-x_2$. Note that this system is not feedback linearizable.

Step 1: Construct an exo-system
\begin{align}
\label{eq:exo_sys_ex}
\dot w = -w-w^2/2
\end{align}
such that $\lambda (w) = -1-w$ and $v(w) = 1$ on $\bR$.

Step 2: Consider the nonlinear Sylvester equation~\eqref{eq:syl}:

\begin{align*}
&\begin{bmatrix}
    -\pi_1(w)-\pi_1^2(w)/2+\pi_1(w)\pi_2(w)+2\pi_2(w)-\pi_2^2(w) \\
    \pi_2(w)-\pi_2^2(w)/2
\end{bmatrix}\\
&\quad+\begin{bmatrix}
    1\\1
\end{bmatrix}l(w)=  (-w-w^2/2)\,\begin{bmatrix}
    \partial \pi_1(w)/\partial w \\ \partial \pi_2(w)/\partial w
\end{bmatrix},
\end{align*}
where $\pi_i(w)$ denotes the $i$-th component of $\pi(w)$. For $l(w) = -2w$, the nonlinear Sylvester equation has a class $C^1$ solution
\begin{align*}
\pi (w) = 
\begin{bmatrix}
        w \\ w
    \end{bmatrix}.
\end{align*}

Step 3: Set 
$k(x) = \varphi(x_2-x_1)-2x_2$, with $\varphi:\mathbb{R}\rightarrow\mathbb{R}$ an arbitrary class $C^1$ function satisfying $\varphi(0)=0$. Clearly, $k(x)$ satisfies \eqref{eq:pep}, and the resulting closed-loop system is
\begin{align}\label{eq:ex2_closed}
    \dot{x} = \begin{bmatrix}
        -x_1-x_1^2/2+x_1x_2-x_2^2+\varphi(x_2-x_1)\\-x_2-x_2^2/2+\varphi(x_2-x_1)
    \end{bmatrix}.
\end{align}
It can be verified that 
\begin{equation*}
\label{eq:ex2_cl_eig}
\begin{aligned}
    \lambda_{c}(x) &= -1-x_2, & 
    v_{c}(x) &= \begin{bmatrix}
    1\\1
\end{bmatrix}
\end{aligned}
\end{equation*}
is a pair of closed-loop nonlinear right eigenvalue and eigenvector on $\bR^2$.
In this case, we achieve partial nonlinear right eigenvalue assignment not only on $\pi(\bR)$ but also on $\bR^2$.

Next, we design a $k(x)$ satisfying \eqref{eq:inv}. As shown in Theorem~\ref{thm:inv}, if $k(x)$ additionally fulfills \eqref{eq:inv}, then the open-loop nonlinear right eigenvalue and its associated nonlinear right eigenvector remain unchanged. This is demonstrated by preserving $\lambda_{o,1}(x) = -1-x_1+x_2$ and $v_{o,1}(x)$.
For $v_{o,1}(x)$, \eqref{eq:inv} reduces to
\begin{align*}
    \left[\begin{bmatrix}
    1\\0
\end{bmatrix},
    \begin{bmatrix}
        1\\1
    \end{bmatrix} k(x)
\right]=0,
\end{align*}
which is equivalent to $$\frac{\partial k(x)}{\partial x_1} = 0.$$
Then, the required controller is 
$
k(x) = -2x_2.
$
The resulting closed-loop system is
\begin{align}
\label{eq:ex3_closed}
    \dot{x} = \begin{bmatrix}
        -x_1-x_1^2/2+x_1x_2-x_2^2\\-x_2-x_2^2/2
    \end{bmatrix}.
\end{align}
The closed-loop nonlinear right eigenvalues and their associated eigenvectors on $\bR^2$ are
    \begin{align*}\label{eq:ex3_cl_eig}
    \lambda_{c,1}(x) &= -1+x_2-x_1 = \lambda_{o,1}(x),& 
    v_{c,1}(x) &= \begin{bmatrix}
    1\\0
\end{bmatrix} =v_{o,1} \nonumber\\ 
\lambda_{c,2}(x) &= -1-x_2,& 
v_{c,2}(x) &= \begin{bmatrix}
    1\\1
\end{bmatrix}.
\end{align*}
Therefore, the open-loop nonlinear right eigenvalue $\lambda_{o,1}(x)$ and its associated right eigenvector $v_{o,1}(x)$ are preserved while achieving partial nonlinear eigenvalue assignment.

\end{example}


\subsection{Remarks on Partial Stabilization}
\label{sec:rem}

For linear systems, eigenvalues characterize stability properties. This is also true in the nonlinear case under the integrability assumption of nonlinear eigenvectors \cite{kawano2015stability}. Therefore, the proposed approach can also be applied to partial stabilization of nonlinear systems, which is illustrated in this subsection. Throughout this subsection, we focus on local convergence to an equilibrium point.

First, we recall a part of the result in \cite[Theorem 1]{astolfi2003immersion}, which establishes the connection between stability of an exo-system~\eqref{eq:exo} and the closed-loop system~\eqref{eq:sys_cl} via the nonlinear Sylvester equation~\eqref{eq:syl}.


\begin{proposition}
\label{prop:ps}
Let $D\subseteq \bR^\nu$ be an open subset containing the origin.
Suppose that there exist class $C^1$ functions $k:\bR^n \to \bR^m$, $\pi:D \to \bR^n$, $l:D \to \bR^m$, and $s:D \to \bR^\nu$ with $k(0) = 0$, $l(0)=0$, $\pi(0)=0$ and $s(0)=0$, satisfying the following two conditions:
\begin{enumerate}
\item each trajectory of the exo-system~\eqref{eq:exo} starting from $D$ is a bounded function of $t \ge 0$ and converges to the origin\footnote{This is weaker than local asymptotic stability of the origin.};

\item the nonlinear Sylvester equation \eqref{eq:syl} and the equality \eqref{eq:pep} hold for all $w\in D$.
\end{enumerate}
Then, for each $x(0) \in \pi (D)$, the corresponding trajectory of the closed-loop system~\eqref{eq:sys_cl} is a bounded function of $t \ge 0$ and converges to the origin.
\end{proposition}
\begin{proof}
The proof can be found in \cite[Theorem 1]{astolfi2003immersion}.
\end{proof}

Hereafter, we call the guaranteed property of the origin of the closed-loop system in Proposition~\ref{prop:ps} \emph{partial stability} for the sake of simplicity. According to this proposition, the origin of the closed-loop system designed by partial nonlinear right eigenvalue assignment is partially stable if the exo-system satisfies item~1) of Proposition~\ref{prop:ps}. Next, we recall a sufficient condition in \cite{kawano2015stability} to verify this by using nonlinear right eigenvalues.

\begin{proposition}
\label{prop:gas}
Let $D\subseteq \bR^\nu$ be an open subset containing the origin.
An exo-system~\eqref{eq:exo} satisfies item 1) of Proposition~\ref{prop:ps} if the following three conditions hold:
\begin{enumerate}
\item it admits $\nu$ class $C^0$ nonlinear right eigenvalues $\lambda_i :  D \to  \bC$ and their associated class $C^2$ nonlinear right eigenvectors $v_i : D \to  \bC^\nu \setminus \{0_{\nu \times 1}\}$, $i=1,\dots,\nu$ such that
\begin{align*}
    &\mathrm{dim}_ \bC(\mathrm{span}_ \bC\{v_1(w),\ldots,v_\nu(w)\}) = \nu \\
    &[v_i(w),v_j(w)]=0, \quad \forall i,j=1,\ldots,\nu
\end{align*}
for all $w \in D$;


\item the origin is an isolated equilibrium point in $D$;


\item $\mathrm{Re}(\lambda_i(w)) \le 0$ holds on $ D$, and $\mathrm{Re}(\lambda_i(w))$ is not identically zero on $D$ for all $i=1 \dots, \nu$.
\end{enumerate}
\end{proposition}
\begin{proof}
Although the domain of the definition of each $v_i$, $i=1,\dots,\nu$ is on $D$ (not $\bC^\nu$), the proof is identical to the argument in
\cite[Remark 2]{kawano2015stability}.
\end{proof}

Item 1) of Proposition~\ref{prop:gas} implies that there exists a change of coordinates $z = \psi (w)$ such that the exo-system becomes diagonal in the $z$-coordinates, i.e., $\dot z_i = \hat s_i (z_i)$, $i=1,\dots,\nu$, where $z_i \in \psi_i(D) \subset  \bC$.
Also, it follows that $\partial \hat s_i (z_i)/\partial z_i = \lambda_i (\psi^{-1}(z))$, meaning that $\lambda_i (\psi^{-1}(z))$ is a nonlinear right eigenvalue in the $z$-coordinates. Items 2) and 3) of Proposition~\ref{prop:gas} guarantee the local asymptotic stability of the origin for each diagonal subsystem. Furthermore, if Proposition~\ref{prop:gas} holds for $D=\bR^\nu$, then every trajectory of each diagonal system converges to the origin.

Now, combining Theorem~\ref{thm:main} with Propositions~\ref{prop:ps} and~\ref{prop:gas}, the following result for partial stabilization follows directly.

\begin{corollary}\label{cor:part}
Let $D\subseteq \bR^\nu$ be an open subset containing the origin.
The origin of the closed-loop system~\eqref{eq:sys_cl} is partially stable
if there exist class $C^1$ functions $k:\bR^n \to \bR^m$, $\pi:D \to \bR^n$, $l:D \to \bR^m$, and $s:D \to \bR^\nu$ with $k(0) = 0$, $l(0)=0$, $\pi(0)=0$ and $s(0)=0$ such that the exo-system~\eqref{eq:exo} satisfies all conditions in Proposition~\ref{prop:gas}.

\end{corollary}





\begin{remark}
    Note that Theorem~\ref{thm:main} requires \eqref{eq:syl} and \eqref{eq:pep} to hold on $\bR^\nu$. Instead, Corollary~\ref{cor:part} only requires this in an open set $D$ around the origin. This can be understood as a condition for partial nonlinear right eigenvalue assignment on $\pi(D)$.
    
\end{remark}

\begin{example}(Continuation of Example~\ref{ex:stab}.)
In this example, we illustrate Corollary~\ref{cor:part} by applying it to partial stabilization of system~\eqref{ex:osys}. For the exo-system~\eqref{eq:exo_sys_ex}, the closed-loop systems~\eqref{eq:ex3_closed} satisfies all the conditions of Corollary~\ref{cor:part} on $D:=\{x\in\bR^2: -1 \le x_2 \le 1 + x_1\}$. Consequently, its origin is partially stable with region of attraction $D$. This suggests that a nonlinear system can be stabilized by shifting only the unstable nonlinear eigenvalues. It is again worth emphasizing that the system~\eqref{ex:osys} is not feedback linearizable.
\end{example}



\section{Partial Left Eigenvalue Assignment}
\label{sec:dual}

In the previous section, we have studied partial nonlinear right eigenvalue assignment. In this section, we consider its dual problem: partial nonlinear left eigenvalue assignment. The results in this section can be applied to partial observer design or partial stabilization by output feedback. 

Consider an autonomous nonlinear system given by
\[
  \dot z  = F(z), \quad
    y = H(z),
\]
where $F:\bR^N \to \bR^N$ and $H:\bR^N \to \bR^M$ are of class $C^1$ such that $F(0) = 0$ and $H(0)=0$. 
As a dual form of the problem of~\eqref{eq:sys_cl}, we study a partial left nonlinear eigenvalue assignment problem of the following system:
\begin{align}\label{eq:sys_ob}
\dot z = \tilde{F}(z):= F(z) + q(z, H(z)),
\end{align}
where $q:\bR^N \times \bR^M \to \bR^N$ is of class $C^1$ such that $q(0,0)=0$.

To address this, we utilize the dual Sylvester equation~\cite{scarciotti2024interconnection}.
Let $r: \bR^\nu \times \bR^M \to \bR^\nu$ and $s:\bR^\nu \to \bR^\nu$ be of class $C^1$ such that $r(0,0)=0$ and $s(0)=0$. The dual Sylvester equation of~\eqref{eq:syl} is the following nonlinear PDE with respect to $\rho (z)$:
\begin{align}
\label{eq:dsyl}
    -s(-\rho (z)) = \frac{\partial \rho (z)}{\partial z} F(z) + r (-\rho (z), H(z)).
\end{align}

Now, we are ready to provide a dual of Theorem~\ref{thm:main} as follows.
\begin{theorem}
\label{thm:dmain}
    Suppose that there exist class $C^1$ functions $q:\bR^N \times \bR^M \to \bR^N$, $r:\bR^\nu \times \bR^M \to \bR^\nu$, and $s:\bR^\nu \to \bR^\nu$ with $q(0,0) = 0$, $r(0,0)=0$, and $s(0)=0$ satisfying the following three conditions:
    \begin{enumerate}
    \item the exo-system~\eqref{eq:exo} admits a nonlinear left eigenvalue $\lambda (w)$ and its associated nonlinear left eigenvector $v(w)$ on $\bR^\nu$;
    
    \item the nonlinear Sylvester equation~\eqref{eq:dsyl} has a class $C^1$ solution $\rho (z)$ on $\bR^N$ such that $\rho (0) = 0$, and 
    \begin{align}
    \label{eq:pep_ob}
        r ( -\rho (z), H(z))=\frac{\partial \rho (z)}{\partial z} q(z, H(z))
    \end{align}
    holds on $\bR^N$;

    \item there exist  $\Tilde{\lambda}:\bR^N \to  \bC$ of class $C^0$ and $\Tilde{v}:\bR^N \to  \bC^n$ of class $C^1$ such that 
   \begin{subequations}\label{cond1:dimm}
    \begin{align}
    \Tilde{\lambda}(z) &= \lambda(- \rho (z) ) \\
    \tilde{v}^\top (z) &= - v^\top (-\rho (z)) \frac{\partial \rho (z)}{\partial z} 
    \label{cond12:dimm}
    \end{align}
    \end{subequations}
    holds on $\bR^N$, and $\tilde{v}^\top(z)\neq0_{1\times N}$ for all $z\in\bR^{N}$
    \end{enumerate}
   Then, $\tilde{\lambda}(z)$ and $\tilde{v}(z)$ are, respectively, a nonlinear left eigenvalue and its associated nonlinear left eigenvector of the closed-loop system \eqref{eq:sys_ob} on $\bR^N$.
\end{theorem}

\begin{proof}
It follows from~\eqref{eq:dsyl}, \eqref{eq:pep_ob}, and \eqref{eq:sys_ob} that 
\begin{align}\label{pf1:dmain}
    -s(-\rho (z)) 
    &= \frac{\partial \rho (z)}{\partial z} F(z) + r (-\rho (z), H(z)) \nonumber\\
    &= \frac{\partial \rho (z)}{\partial z} (F(z) + q(z, H(z))) \nonumber\\
    &= \frac{\partial \rho (z)}{\partial z} \tilde F(z).
\end{align}
Taking the partial derivative of both side with respect to $z$ yields
\begin{align}\label{pf2:dmain}
&\frac{\partial \rho (z)}{\partial z} \frac{\partial \tilde F(z)}{\partial z} \nonumber\\
&= - \frac{\partial s(-\rho (z))}{\partial z}  - \sum_{i=1}^n \frac{\partial}{\partial z_i} \left( \frac{\partial \rho (z)}{\partial z} \right) \tilde F_i (z).
\end{align}

Based on them, we show the statement.
First, we have, from~\eqref{cond12:dimm} and~\eqref{pf2:dmain}, 
\begin{align}\label{pf3:dmain}
&\tilde v^\top (z) \frac{\partial \tilde F (z)}{\partial z} \nonumber\\
&= - v^\top (-\rho (z)) \frac{\partial \rho (z)}{\partial z} \frac{\partial \tilde F (z)}{\partial z} \nonumber\\
&= v^\top (-\rho (z))\frac{\partial s(-\rho (z))}{\partial z} \nonumber\\
&\qquad + v^\top (-\rho (z)) \sum_{i=1}^n \frac{\partial}{\partial z_i} \left( \frac{\partial \rho (z)}{\partial z} \right) \tilde F_i (z).
\end{align}

Next, utilizing~\eqref{cond12:dimm}, compute
\begin{align}\label{pf4:dmain}
\frac{\partial \tilde v(z)}{\partial z} \tilde F(z)
&= - \left(\sum_{i=1}^n \frac{\partial}{\partial z_i}\left(\frac{\partial^\top \rho (z)}{\partial z}\right) v_i (-\rho (z))\right)  \tilde F(z) \nonumber\\
&\qquad - \frac{\partial^\top \rho (z)}{\partial z} \frac{\partial v (-\rho(z))}{\partial z} \tilde F(z).
\end{align}
We focus on the second term of the right-hand side. Applying~\eqref{pf1:dmain} and \eqref{eq:leig} in order gives
\begin{align*}
\frac{\partial v (-\rho(z))}{\partial z} \tilde F(z)
&= - \left.\frac{\partial v (w)}{\partial w}\right|_{w = -\rho (z)} \frac{\partial \rho(z)}{\partial z} \tilde F(z)\\
&= \left. \left(\frac{\partial v (w)}{\partial z} s(w)\right) \right|_{w=-\rho(z)} \\
&= \left. \left(\lambda (w) v (w) - \frac{\partial^\top s(w)}{\partial w} v(w))\right) \right|_{w=-\rho(z)}.
\end{align*}
Substituting this into~\eqref{pf4:dmain} and utilizing~\eqref{cond1:dimm} lead to
\begin{align*}
&\frac{\partial \tilde v(z)}{\partial z} \tilde F(z) \\
&= - \left(\sum_{i=1}^n \frac{\partial}{\partial z_i}\left(\frac{\partial^\top \rho (z)}{\partial z}\right) v_i (-\rho (z))\right)  \tilde F(z) \\
&\qquad - \frac{\partial^\top \rho (z)}{\partial z} \left. \left(\lambda (w) v (w) - \frac{\partial^\top s(w)}{\partial w} v(w))\right) \right|_{w=-\rho(z)} \\
&= - \left(\sum_{i=1}^n \frac{\partial}{\partial z_i}\left(\frac{\partial^\top \rho (z)}{\partial z}\right) v_i (-\rho (z))\right)  \tilde F(z) \\
&\qquad + \tilde \lambda (z) \tilde v (z) - \frac{\partial^\top s(-\rho(z))}{\partial z} v(-\rho(z)).
\end{align*}
From this and \eqref{pf3:dmain}, we have
\begin{align*}
&\tilde v^\top (z) \frac{\partial \tilde F(z)}{\partial z} + \left(\frac{\partial \tilde v(z)}{\partial z} \tilde F(z)\right)^\top \\
&= \tilde \lambda (z) \tilde v^\top (z) + v^\top (-\rho (z)) \sum_{i=1}^n \frac{\partial}{\partial z_i} \left( \frac{\partial \rho (z)}{\partial z} \right) \tilde F_i (z)\\
&\qquad 
- \tilde F^\top (z) \left(\sum_{i=1}^n \frac{\partial}{\partial z_i}\left(\frac{\partial^\top \rho (z)}{\partial z}\right) v_i (-\rho (z))\right)^\top \\
&= \tilde \lambda (z) \tilde v^\top (z).
\end{align*}
This completes the proof.
\end{proof}
\begin{remark}
    A sufficient condition for $\tilde{v}^\top(z)\neq0_{1\times N}$ is given by $\mathrm{rank}(\partial \rho(z)/\partial z)=\nu, \ \forall z\in\bR^{N}$.
\end{remark}

We next provide a condition for preserving an open-loop nonlinear left eigenvalue and eigenvector.
\begin{theorem}
Let $\lambda:\bR^N \to  \bC$ of class $C^0$ and $v:\bR^N \to \bC^N\setminus\{0_{N\times1}\}$ of class $C^1$ be, respectively, a nonlinear left eigenvalue and its associated nonlinear left eigenvector of $\dot z = F(z)$ on $\bR^N$. If $q:\bR^N \times \bR^M \to \bR^N$ of class $C^1$ satisfies
    \begin{align}
    \label{eq:dinv}
        v^\top (z)\frac{\partial q(z,H(z))}{\partial z} + \left(\frac{\partial v(z)}{\partial z}q(z,H(z))\right)^\top = 0
    \end{align}
    on $\bR^N$, then $\lambda (z)$ and $v(z)$ are, respectively, a nonlinear left eigenvalue and its associated nonlinear left eigenvector of the system~\eqref{eq:sys_ob} on $\bR^N$.
\end{theorem}
\begin{proof}
It follows from~\eqref{eq:leig}, i.e.,
\begin{align*}
v^\top (z)\frac{\partial F(z)}{\partial z} + \left(\frac{\partial v(z)}{\partial z}F(z)\right)^\top = \lambda(z) v^\top (z)
\end{align*}
and~\eqref{eq:dinv} that
\begin{align*}
&v^\top (z)\frac{\partial (F(z) + q(z,H(z)))}{\partial z} \\
&\qquad + \left(\frac{\partial v(z)}{\partial z}(F(z) + q(z,H(z)))\right)^\top  
 = \lambda (z) v^\top (z).
\end{align*}
That completes the proof.
\end{proof}

We apply Theorem~\ref{thm:dmain} to partial observer design. 
Consider an autonomous nonlinear system, given by
\[  \dot x = f(x), \quad
    y = h(x),
    \]
where $f:\bR^n \to \bR^n$ and $h:\bR^n \to \bR^m$ are of class $C^1$ such that $f(0) = 0$ and $h(0)=0$. As an observer dynamics, we consider the following system 
\begin{align}\label{eq:ob}
\dot \xi = f(\xi) + p(\xi, y),
\end{align}
where a class $C^1$ function $p:\bR^n \times \bR^m \to \bR^n$ is such that $p(x,h(x))=0$.

Introducing the error $e := \xi - x$, the interconnected system can be represented as~\eqref{eq:sys_ob} with $z := [\begin{matrix}x & e\end{matrix}]^\top$ and
\begin{subequations}\label{eq:Fq}
\begin{align}
F(z) &:= \begin{bmatrix}f(x) \\ f(x + e) - f(x)\end{bmatrix} \label{eq:F}\\
q(z,y) &:= \begin{bmatrix}0_{n\times 1}\\ p(x + e, y) \end{bmatrix}, \quad
H(z) := \begin{bmatrix}h(x) \\ 0_{m\times 1} \end{bmatrix}.
\end{align}
    \end{subequations}
Then, \eqref{eq:ob} becomes a partial observer if $p(\xi, y)$ is defined such that a part of $e$-dynamics converges to the origin. This is a partial stabilization problem of~\eqref{eq:sys_ob}. Nonlinear left eigenvalues can be used for partial stability analysis of~\eqref{eq:sys_ob}, by reasoning similar to that mentioned for nonlinear right eigenvalues in Section~\ref{sec:rem}.

\begin{example}
We consider the observer design of the following nonlinear system:
\begin{align*}
\left\{\begin{alignedat}{2}
    \dot{x} &= f(x) = \begin{bmatrix}
        -x_1-2x_2-3x_2^2\\
        x_2
    \end{bmatrix}\\
y&= h(x) =x_2.
\end{alignedat}\right.
\end{align*}
For this system, $F(z)$ in~\eqref{eq:F} is
\begin{align*}
F(z) = \begin{bmatrix}
        -x_1-2x_2-3x_2^2\\
        x_2\\
        -e_1-2e_2-6x_2e_2-3e_2^2\\
        e_2
    \end{bmatrix}.
\end{align*}
A nonlinear left eigenvalue and its associated nonlinear left eigenvector of the $e$-dynamics part of the system $\dot z = F(z)$ is 
\begin{align*}
\lambda_{o,1}(z) = -1, \quad
v_{o,1}^\top(z) = \begin{bmatrix}
    0 & 0& 1 & -1 - 2e_2
\end{bmatrix}
\end{align*}
on $\bR^4$.
The objective is to design $p(\xi, y)$ such that $p(x,h(x))=0$, while preserving the nonlinear left eigenvalue $\lambda_{o,1}$ and assigning the other closed-loop nonlinear left eigenvalue of the error dynamics to $-1$.

Step 1: To achieve this, we introduce an exo-system:
\begin{align*}
\dot w = s(w) =  -w
\end{align*}
with a stable nonlinear eigenvalue $\lambda (w) = -1$ and eigenvector $v(w) = 1$ on $\bR$.

Step 2: Substituting~\eqref{eq:pep_ob} and~\eqref{eq:Fq} into the dual Sylvester equation~\eqref{eq:dsyl} leads to
\begin{align*}
    -\rho(z) 
    &= \frac{\partial \rho (z)}{\partial z}\begin{bmatrix}
        -x_1-2x_2-3x_2^2\\
        x_2\\
        -e_1-2e_2-6x_2e_2-3e_2^2\\
        e_2
    \end{bmatrix}\nonumber\\
    &\qquad + \frac{\partial \rho (z)}{\partial z}\begin{bmatrix}
        0_{2\times 1}\\
        p(x+e, h(x))
    \end{bmatrix}.
\end{align*}
To preserve the open-loop nonlinear left eigenvalue $\lambda_{o,1}(z)$, we substitute $\lambda_{o,1}(z)$ and $v_{o,1}(z)$ into \eqref{eq:dinv}, which yields
\begin{align*}
&\begin{bmatrix}
    0 & 0& 1 & -1 - 2e_2
\end{bmatrix}
\begin{bmatrix}0_{2\times2} & 0_{2\times2} \\ \displaystyle \frac{\partial p(x + e, h(x))}{\partial x} & \displaystyle \frac{\partial p(x + e, h(x))}{\partial e} \end{bmatrix} \nonumber\\
&+ \begin{bmatrix}
        0 & 0 & 0 & p_2(x+e, h(x))
    \end{bmatrix} = 0_{1\times 4}.
\end{align*}

Step 3: Set $p (\xi, y)$ as 
\begin{align*}
p(\xi,y) = 
\begin{bmatrix}
4\xi_2^2+2\xi_2-2y\xi_2-2y^2-2y\\
-2\xi_2+2y,
\end{bmatrix}
\end{align*}
which satisfies $p(x,h(x))= 0$ and the dual Sylvester equation with  
$
\rho(z) = \begin{bmatrix}
    0 & 0  &e_1+e_2^2 & e_2
\end{bmatrix}^\top.
$
This design results in the closed-loop system \eqref{eq:sys_ob} as
\begin{align*}
    \begin{bmatrix}
        \dot{x}_1\\\dot{x}_2\\\dot{e}_1\\\dot{e}_2
    \end{bmatrix} = \begin{bmatrix}
    -1-2x_2-3x_2^2\\
    x_2\\
    -e_1+e_2^2\\
    -e_2
\end{bmatrix},
\end{align*}
which now possesses the nonlinear left eigenstructure:
\begin{align*}
\lambda_{c,1}(z) &= -1, & 
    v_{c,1}^\top(z) &= \begin{bmatrix}
    0&0&1&-1-2e_2
\end{bmatrix}\\
\lambda_{c,2}(z) &= -1, & 
    v_{c,2}^\top(z) &= \begin{bmatrix}
    0&0&0&1
\end{bmatrix}.
\end{align*}
Thus, partial nonlinear left eigenvalue assignment is achieved while preserving the open-loop nonlinear left eigenvalue $\lambda_{o,1}(x)$ and the associated nonlinear left eigenvector $v_{o,1}^\top(x)$. 

In addition, since $v_{c,1}(z)$ and $v_{c,2}(z)$ are linearly independent and integrable on $\bR^4$, the $e$-dynamics converge to zero as $t \rightarrow \infty$. The resulting observer is
\begin{align*}
    \dot{\xi} &= \begin{bmatrix}
        -\xi_1-2\xi_2-3\xi_2^2\\
        \xi_2
    \end{bmatrix}  
      + 2 \begin{bmatrix}
2\xi_2^2+\xi_2-y\xi_2-y^2-y\\
-\xi_2+y
    \end{bmatrix}.
\end{align*}
Figure~\ref{fig:ob} shows the trajectories of the system and the designed observer, which shows $\xi \to x$ as $t \to \infty$. 

\begin{figure}[t]
    \centering
    \includegraphics[width=0.45\textwidth]{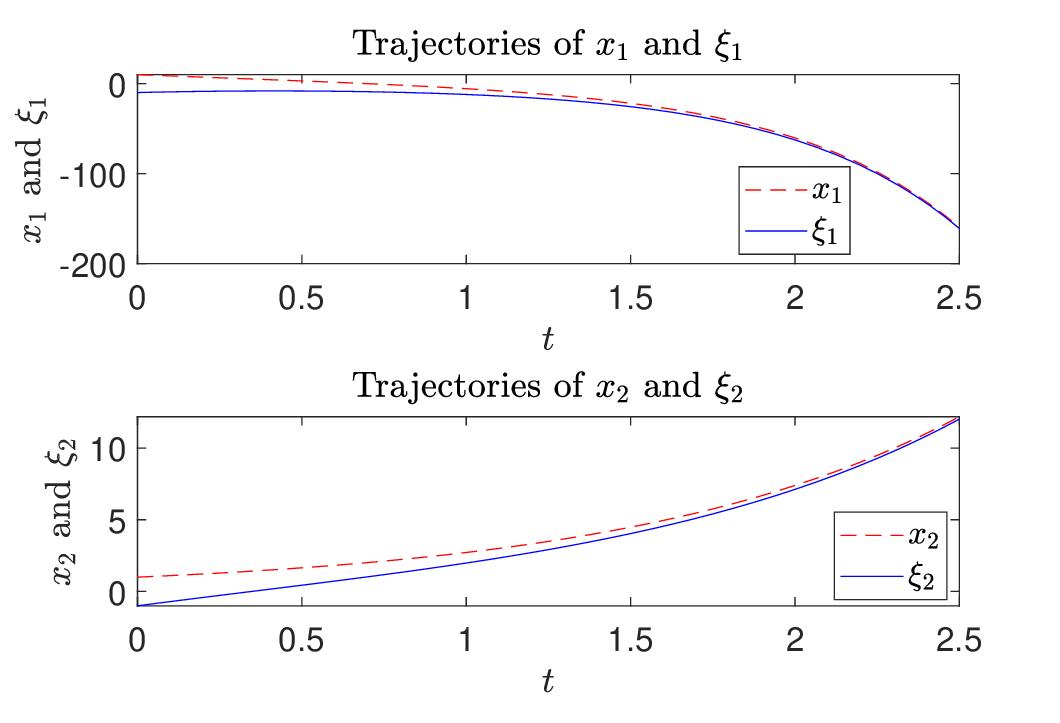}
    \caption{Trajectories of the states of the system $x$ and its observer $\xi$}
    \label{fig:ob}
\end{figure}

\end{example}



\section{Conclusion}
\label{sec:conclusion}
In this paper, we have proposed novel methods for partial nonlinear eigenvalue assignment of nonlinear systems using nonlinear Sylvester equations derived from immersion conditions. First, we have addressed a partial nonlinear eigenvalue assignment problem related to partial stabilization via state feedback. In particular, we have shown that a subset of the nonlinear right eigenvalues of the closed-loop system can be assigned to a set of specified scalar-valued functions by solving a nonlinear Sylvester equation. We have also derived conditions under which an open-loop nonlinear eigenvalue remains unchanged. 
Then, we studied the dual problem of left nonlinear eigenvalue assignment, showing that a subset of nonlinear left eigenvalues can likewise be assigned through a dual Sylvester equation.

\addtolength{\textheight}{-12cm}   






\bibliographystyle{IEEEtran}
\bibliography{mybib}

\end{document}